\documentclass[11pt, reqno, psamsfonts]{amsart}
\pdfoutput=1

\usepackage{amssymb}
\usepackage{amsthm}
\usepackage{amsmath}
\usepackage{latexsym}
\usepackage[T1]{fontenc}
\usepackage[utf8]{inputenc}
\usepackage[russian, french, english]{babel}

\usepackage{graphicx}
\usepackage{wrapfig}
\usepackage[justification=centering, labelfont=bf]{caption}
\usepackage{mathtools}
\usepackage{amsbsy}
\usepackage[inline]{enumitem}
\usepackage{mathrsfs}
\usepackage{array}
\usepackage{multicol}
\usepackage{stmaryrd}
\usepackage{cancel}
\usepackage{lmodern}
\usepackage{mathabx}
\usepackage{upgreek}
\usepackage{titlesec}
\usepackage{titletoc}
\usepackage[spacing=true,kerning=true,babel=true,tracking=true]{microtype}
\usepackage[shortcuts]{extdash}
\usepackage[foot]{amsaddr}
\usepackage[left=1in,right=1in,top=1in,bottom=1in,bindingoffset=0cm]{geometry}
\usepackage{bm}
\usepackage{centernot}
\usepackage{mdframed}
\usepackage[hidelinks]{hyperref}
\usepackage{xspace}
\usepackage[skins]{tcolorbox}
\usepackage{framed}
\usepackage[justification=centering, labelfont=bf]{caption}
\usepackage{tikz}
\usetikzlibrary{shapes,snakes}
\usetikzlibrary{arrows.meta}
\usetikzlibrary{decorations.pathmorphing}
\usetikzlibrary{patterns}
\usepackage{float}

\usepackage[
    sortcites,
    backend=biber, style=alphabetic, sorting=nyt, maxnames=100,backref=true]{biblatex}
\title{\sffamily Strong marker sets for arbitrary generating sets of $\Z^n$}%
\date{}

\author{Jing~Yu}
\address{\normalfont Shanghai Center for Mathematical Sciences, Fudan University, Shanghai, China. Email: \texttt{jyu@fudan.edu.cn}.}
\thanks{JY's research is partially supported by the National Natural Science Foundation of China grants 12371343 and 12525110 (PI: Hehui Wu).}

\newtheoremstyle{bfnote}%
{}{}%
{\slshape}{}%
{\bfseries}{\bfseries.}%
{ }%
{\thmname{#1}\thmnumber{ #2}\thmnote{ \ep{\normalfont{}#3}}}

\theoremstyle{bfnote}
\newtheorem{theo}{Theorem}[section]
\newtheorem*{theo*}{Theorem}

\newtheorem{lemma}[theo]{Lemma}

\newtheorem*{corl*}{Corollary}

\theoremstyle{definition}

\newtheorem*{defn*}{Definition}

\newtheorem{remk}[theo]{Remark}

\newtheorem*{exmp*}{Example}

\theoremstyle{remark}
\newtheorem*{ques*}{Question}
\newtheorem*{remk*}{Remark}

\newcommand*{\myproofname}{Proof}

\makeatletter
\newcommand{\neutralize}[1]{\expandafter\let\csname c@#1\endcsname\count@}
\makeatother

\newcommand{\Z}{\mathbb{Z}}

\renewcommand{\epsilon}{\varepsilon}

\renewcommand{\phi}{\varphi}
\renewcommand{\theta}{\vartheta}
\renewcommand{\leq}{\leqslant}

\newcommand{\bemph}[1]{{\normalfont#1}} 
\newcommand{\ep}[1]{\bemph{(}#1\bemph{)}} 

\newcommand{\emphd}[1]{{\fontseries{b}\selectfont\textsf{#1}}}

\newcommand{\Free}{\mathsf{Free}}

\newcommand{\SL}{\mathrm{SL}}

\numberwithin{equation}{section}

\titleformat{\section}[block]{\large\bfseries\sffamily}{\thesection.}{1ex}{}
\titleformat{\subsection}[block]{\bfseries\sffamily}{\thesubsection.}{1ex}{}
\titleformat{\subsubsection}[block]{\itshape}{\bfseries\upshape\sffamily\thesubsubsection.}{1ex}{}

\titlespacing*{\section}{0pt}{*3}{*1}
\titlespacing*{\subsection}{0pt}{*3}{*1}
\titlespacing*{\subsubsection}{0pt}{*2}{*1}

\titlecontents{section}
 [1.5em] %
 {\smallskip}
 {\bfseries\thecontentslabel\hspace{1.02em}}
{\bfseries}
 {\,\,\titlerule*[0.77pc]{}\bfseries\contentspage}
\titlecontents{subsection}
 [4em] %
 {\smallskip}
 {\thecontentslabel\hspace{1.02em}}
{\hspace*{2.32em}}
 {\,\,\titlerule*[0.77pc]{.}\contentspage}

\renewbibmacro{in:}{}

\renewbibmacro*{volume+number+eid}{%
	\printfield{volume}%
	\setunit*{\addnbspace}
	\printfield{number}%
	\setunit{\addcomma\space}%
	\printfield{eid}}

\DeclareFieldFormat[article]{volume}{\textbf{#1}\space}
\DeclareFieldFormat[article]{number}{\mkbibparens{#1}}

\DeclareFieldFormat{journaltitle}{#1,}
\DeclareFieldFormat[thesis]{title}{\mkbibemph{#1}\addperiod}
\DeclareFieldFormat[article, unpublished, thesis]{title}{\mkbibemph{#1},}
\DeclareFieldFormat[book]{title}{\mkbibemph{#1}\addperiod}
\DeclareFieldFormat[unpublished]{howpublished}{#1, }

\DeclareFieldFormat{pages}{#1}

\DeclareFieldFormat[article]{series}{Ser.~#1\addcomma}

\setlist{topsep=3pt,itemsep=3pt}

\begin{document}


    \maketitle
    
    \vspace*{-5pt}
    
    \begin{abstract}
       Gao and Wang proved a strong clopen marker theorem for finite generating sets of $\Z^n$ under the assumption that each generator has support of size either $1$ or $n$. We show that this support assumption can be removed. The proof is a short conjugacy argument: after a unimodular change of coordinates, any finite set of nonzero lattice vectors can be put in full-support position, allowing one to apply the theorem of Gao and Wang and conjugate the resulting marker set back.
    \end{abstract}

\section{Introduction}
Let $\Z^n$ act on $2^{\Z^n}$ by the \emphd{Bernoulli shift}
\[
 (g\cdot x)(h)=x(g+h),\qquad g,h \in \Z^n.
\]
The \emphd{free part} of $2^{\Z^n}$, denoted by $\Free(2^{\Z^n})$, is the set of all $x \in 2^{\Z^n}$ whose stabilizer is trivial.

On each free orbit, we identify the orbit with $\Z^n$. 
Thus, if $x, y \in \Free(2^{\Z^n})$ lie in the same orbit and let $y=g\cdot x$ for the unique $g \in \Z^n$, we define \[\rho(x,y):=\|g\|_\infty:=\max_{1\le i\le n}|g_i|.\]

Gao and Wang proved the following strong clopen marker theorem for certain generating sets:

 \begin{theo}[{\cite[Theorem 1.3]{GW}}]\label{thm:GW} 
    Let $n, d \ge 1$ be positive integers and let $S\subseteq \Z^n$ be a finite generating set. Suppose for each $g\in S$, the set $\{1\le i\le n\colon g_i\ne 0\}$ has size either $1$ or $n$. Then there are a positive integer $\Delta\ge d$ and a clopen subset $M \subseteq \Free(2^{\Z^n})$ such that:
    \begin{enumerate}
    \item if $x,y\in M$ are distinct and lie in the same orbit, then $\rho(x,y)\ge d$;
    \item for any $g\in S$ and any $x\in \Free(2^{\Z^n})$, there are non-negative integers $a, b\leq \Delta$ such that $ag\cdot x\in M$ and $-bg\cdot x\in M$.
    \end{enumerate}
\end{theo}

The purpose of this note is to remove the support assumption on the generators. The condition $0\notin S$ in the statement below is natural: if $0\in S$, the second condition for $g=0$ would force $x\in M$ for every $x$, which is incompatible with the intended separation condition when $d>1$.

\begin{theo}\label{thm:main}
    Let $n,d$ be positive integers and let $S\subseteq\Z^n\setminus\{0\}$ be a finite generating set.  Then there are a positive integer $\Delta\ge d$ and a clopen set
    $M\subseteq \Free(2^{\Z^n})$
such that:
\begin{enumerate}
 \item if $x,y\in M$ are distinct and lie in the same orbit, then $\rho(x,y)\ge d$;
 \item for every $g\in S$ and every $x\in \Free(2^{\Z^n})$, there are non-negative integers $a,b\le \Delta$ such that
$ag\cdot x\in M$ and $-bg\cdot x\in M$.
\end{enumerate}
\end{theo}

The proof is based on the elementary observation that a finite set of nonzero lattice vectors can be put in full-support position by a unimodular change of coordinates. The only point requiring care is that the metric $\rho$, which is defined using the standard $\ell^\infty$-norm, is not invariant under a general element of $\SL_n(\Z)$. We compensate for this by applying Theorem~\ref{thm:GW} with a larger separation parameter.

For an integer matrix $A = (a_{ij})_{n \times n}$, write \[\|A\|_{\infty} = \max_{1 \le i\le n}\sum_{j=1}^n|a_{ij}|.\] 
Thus
\[
\|Ah\|_\infty\le\|A\|_\infty\|h\|_\infty
\qquad\text{for all }h\in\Z^n.
\]


\section{A unimodular general-position lemma}

\begin{lemma}\label{lem}
Let $F\subseteq\Z^n\setminus\{0\}$ be finite. Then there is a matrix 
$A\in\SL_n(\Z)$ such that for every $v\in F$ and every $1\le j\le n$, 
we have $(Av)_j\ne 0$.
Moreover, such a matrix $A$ can be constructed explicitly from $F$.
\end{lemma}

\begin{proof}
If $F=\emptyset$, take $A=I_n$. The case $n=1$ is immediate, so assume 
$n\ge 2$. Let
\[
C=\max\{\|v\|_\infty:v\in F\}.
\]
Choose integers $t>C+1$ and $q>C$, and put
\[
w=(1,t,t^2,\ldots,t^{n-1})\in\Z^n .
\]
We first claim that $\langle w,v\rangle\ne0$ for every $v\in F$. Fix 
$v=(v_1,\ldots,v_n)\in F$, and let $k$ be maximal such that $v_k\ne0$. If 
$k=1$, then $\langle w,v\rangle=v_1\ne0$. If $k>1$, then
\[
\left|\sum_{i<k}v_it^{i-1}\right|
\le C\sum_{i=0}^{k-2}t^i
=
C\frac{t^{k-1}-1}{t-1}
<
t^{k-1}
\le |v_k|t^{k-1}.
\]
Hence the highest nonzero term cannot be cancelled by the lower order terms, so
$\langle w,v\rangle\ne0$.

Now define $A$ by taking its rows to be
\[
r_1=w,\qquad r_j=e_j+qw \quad (2\le j\le n),
\]
where $e_j$ denotes the $j$th standard basis vector.
The matrix with rows $w,e_2,\ldots,e_n$ is upper triangular with diagonal entries
all equal to $1$, and hence has determinant $1$. Passing from this matrix to $A$
amounts to adding $q$ times the first row to each of the rows $2,\ldots,n$, so the
determinant remains $1$. Thus $A\in\SL_n(\Z)$.

Finally, for every $v\in F$,
\[
(Av)_1=\langle w,v\rangle\ne0.
\]
For $2\le j\le n$,
\[
(Av)_j=v_j+q\langle w,v\rangle.
\]
Since $\langle w,v\rangle$ is a nonzero integer and $q>C\ge |v_j|$, this is nonzero. Therefore every coordinate of every $Av$, $v\in F$, is nonzero.
\end{proof}

\section{Proof of Theorem~\ref{thm:main}}
\begin{proof}[Proof of Theorem~\ref{thm:main}]
    Let $S\subseteq\Z^n\setminus\{0\}$ be a finite generating set and fix $d \ge 1$.
    Apply Lemma~\ref{lem} to $F = S$. We get $A\in\SL_n(\Z)$ such that every vector in $AS = \{Ag\,:\, g\in S\}$ has no zero coordinate. Since $A$ is an automorphism of the lattice $\Z^n$, the set $AS$ is again a finite generating set of $\Z^n$. Moreover, every element of $AS$ has support of size $n$, so $AS$ satisfies the hypothesis in Theorem~\ref{thm:GW}. 

    Recall that $\|Ah\|_{\infty} \le \|A\|_{\infty} \|h\|_{\infty}$ for all $h \in \Z^n$. Apply Theoerm~\ref{thm:GW} with $AS$ and $d\|A\|_{\infty}$ in place of $S$ and $d$, 
    we can get some clopen set $M' \subseteq \Free(\Z^n)$ and an integer $\Delta \ge d\|A\|_{\infty}$ such that
    \begin{enumerate}
 \item if $x,y\in M'$ are distinct and lie in the same orbit, then $\rho(x,y)\ge d\|A\|_{\infty}$;
 \item for every $g\in AS$ and every $x\in \Free(2^{\Z^n})$, there are non-negative integers $a,b\le \Delta$ such that
$ag\cdot x\in M'$ and $-bg\cdot x\in M'$.
\end{enumerate}
Define
    \[\Phi_A: \Free(2^{\Z^n}) \to \Free(2^{\Z^n})\] by 
    \[\Phi_A(x)(h) = x(A^{-1}h), \qquad h \in \Z^n.\] 
    Since $A^{-1}\in\SL_n(\Z)$, the map $\Phi_A$ is a coordinate permutation of $2^{\Z^n}$, with inverse $\Phi_{A^{-1}}$.
    For every $g\in\Z^n$ and every $x\in2^{\Z^n}$, we have 
    \[\Phi_A(g \cdot x) = Ag \cdot \Phi_A(x).\]
   Indeed, evaluate both sides at $h \in \Z^n$ gives 
    \[\Phi_A(g \cdot x) = (g \cdot x)(A^{-1}h) = x(g+A^{-1}h),\]
    while
    \[(Ag \cdot \Phi_A(x))(h) = \Phi_A(x)(Ag+h) = x(A^{-1}(Ag+h)) = x(g+A^{-1}h). \]
    It follows in particular that $\Phi_A$ maps $\Free(2^{\Z^n})$ onto itself. Indeed,
if $x\in\Free(2^{\Z^n})$ and $u\cdot\Phi_A(x)=\Phi_A(x)$, then
\[
\Phi_A(A^{-1}u\cdot x)=\Phi_A(x),
\]
so $A^{-1}u\cdot x=x$. Since $x$ is free, $A^{-1}u=0$, and hence $u=0$.

Now define $M = \Phi_A^{-1}(M')$. Since $M'$ is clopen and $\Phi_A$ is a homeomorphism on $\Free(2^{\Z^n})$, $M$ is clopen. 

We verify the separation condition. Suppose $x, y \in M$ are distinct and lie in the same orbit. Write $y = h \cdot x$ with $0 \neq h \in \Z^n$. Then 
\[\Phi_A(y) = Ah\cdot\Phi_A(x).\] 

Since $\Phi_A(x), \Phi_A(y) \in M'$, the separation property of $M'$ gives
\[\|Ah\|_{\infty} \ge d\|A\|_{\infty}.\]
On the other hand,
\[\|Ah\|_{\infty} \le \|A\|_{\infty}\|h\|_{\infty},\]
Since $\|A\|_{\infty} > 0$,  it follows that $\|h\|_{\infty} \ge d$, thus $\rho(x, y) \ge d$.

Finally, fix $g \in S$ and $x \in \Free(2^{\Z^n})$. Since $Ag \in AS$, the second property of $M'$ gives non-negative integers $a, b \le \Delta$ such that 
\[a(Ag)\cdot \Phi_A(x) \in M' \quad\text{and}\quad -b(Ag) \cdot \Phi_A(x) \in M'.\]
Equivalently, 
\[\Phi_A(ag\cdot x) \in M'\quad\text{and}\quad\Phi_A(-bg\cdot x) \in M'.\]
Hence
\[ag\cdot x \in M\quad\text{and}\quad -bg\cdot x \in M.\]
Since $\Delta \ge d\|A\|_{\infty} \ge d$, this completes the proof.
\end{proof}

\begin{remk}
The proof above is constructive in the following sense. The matrix $A$ in Lemma~\ref{lem} is obtained explicitly from the finite set $S$. Thus, together with the constructive proof of Theorem~\ref{thm:GW}, the argument gives a constructive proof of Theorem~\ref{thm:main}.
\end{remk}
\section*{Acknowledgements}
The author thanks Zhaokuan Hao for his encouragement to finish this paper. 

\printbibliography
\end{document}